\documentclass[11pt]{article}

\usepackage[latin1]{inputenc}
\usepackage[english]{babel}
\usepackage{here,enumerate,graphicx}
\usepackage{amsmath,amssymb,amscd,amsthm,mathrsfs}

\usepackage[flushmargin]{footmisc}
\newcommand\blfootnote[1]
{\begingroup\renewcommand\thefootnote{}\footnote{#1}\addtocounter{footnote}{-1}\endgroup}

\theoremstyle{plain}
\newtheorem{theorem}{Theorem}[section]

\newtheorem{lemma}[theorem]{Lemma}
\newtheorem{corollary}[theorem]{Corollary}

\newtheorem{fact}{Fact}[subsection]

\theoremstyle{definition}
\newtheorem{definition}[theorem]{Definition}

\newtheorem{question}[theorem]{Question}
\newtheorem{problem}[theorem]{Problem}

\newcommand{\NN}{\mathbb{N}}
\newcommand{\ZZ}{\mathbb{Z}}
\newcommand{\RR}{\mathbb{R}}
\newcommand{\CC}{\mathbb{C}}
\newcommand{\KK}{\mathbb{K}}

\newcommand{\Lc}{\mathcal{L}}

\newcommand{\cl}{\overline}
\newcommand{\lspan}{\operatorname{span}}
\newcommand{\res}{\vert}
\newcommand{\eps}{\varepsilon}

\newcommand{\sbf}{\boldsymbol}

\begin{document}
\begin{center}
	\begin{LARGE}
		{\bf Shifts on trees versus classical shifts in chain recurrence}
	\end{LARGE}
\end{center}

\begin{center}
	\begin{Large}
		Antoni L\'opez-Mart\'inez and Dimitris Papathanasiou\blfootnote{\textbf{Journal-ref}: Journal of Differential Equations, Volume 433, article number 113230, (2025).\\ \textbf{DOI}: https://doi.org/10.1016/j.jde.2025.113230}
	\end{Large}
\end{center}


\begin{abstract}
	We construct continuous (and even invertible) linear operators acting on Banach (even Hilbert) spaces whose restrictions to their respective closed linear subspaces of chain recurrent vectors are not chain recurrent operators. This construction completely solves in the negative a problem posed by Nilson C.\ Bernardes Jr.\ and Alfred Peris on chain recurrence in Linear Dynamics. In particular: we show that the non-invertible case can be directly solved via relatively simple weighted backward shifts acting on certain unrooted directed trees; then we modify the non-invertible counterexample to address the invertible case, but falling outside the class of weighted shift operators; and we finally show that this behaviour cannot be achieved via classical (unilateral neither bilateral) weighted backward sifts (acting on $\mathbb{N}$ and $\mathbb{Z}$ respectively) by noticing that a classical shift is a chain recurrent operator as soon as it admits a non-zero chain recurrent vector.
\end{abstract}


\section{Introduction}

This paper concerns Linear Dynamics, that is, the branch of mathematics which studies the orbits generated by the iterations of {\em linear operators} acting on {\em topological vector spaces}, combining the classical field of Topological Dynamics with those of Functional Analysis and Operator Theory. The origin of this research area lies in the notion of {\em hypercyclicity} (existence of a dense orbit), and from there many different lines of work have been derived such as {\em frequent hypercyclicity}, {\em mixing properties}, {\em chaos} and {\em disjoint-hypercyclicity} to name a few (see the textbooks \cite{BaMa2009_book,GrPe2011_book} and the 2020 survey \cite{Gilmore2020} for a comprehensive compilation of the theory developed in this area along the last 40 years).

Although Linear Dynamics has always been at the forefront of several mathematical domains, there are at least two reasons justifying why hypercyclicity has historically been the primary focus in this area: firstly, the existence of a dense orbit for a linear map can be considered as an exotic or even rare behaviour; and secondly, because of the very famous ``{\em Invariant Subspace Problem}'', whose strengthened ``{\em Invariant Subset Problem}''-version can be formulated in terms of the density of every non-zero orbit (see for instance \cite[Section~2.1]{Gilmore2020}). This great emphasis in the notion of hypercyclicity has led to a neglect of other intriguing dynamical properties, particularly those originating in the realm of {\em differentiable} and {\em compact dynamical systems}, that have just received attention within the linear context over the last decade and which can be considered as ``{\em new}'' in Linear Dynamics.

Between these properties we can find that of {\em linear recurrence}, introduced in the 2014 paper \cite{CoMaPa2014} and followed by \cite{BoGrLoPe2022,GriLoPe2025_AMP,Lopez2024_IMRN,LoMe2025_JMAA} among others; the study of notions coming from Ergodic Theory such as {\em linear measure preserving systems} (see \cite{GriLo2023,GriMa2014,Lopez2024_RinM}) or that of {\em entropy} for certain classes of operators (see the 2020 paper \cite{BriKe2020}); but also the concepts of {\em expansivity} and {\em hyperbolicity}, which have been strongly connected to those of {\em chain recurrence} and {\em shadowing} (with its multiple variations) for linear operators in the very recent works \cite{AlBerMess2021,AntMantVar2022,BerCiDaMessPu2018,BerMe2020,BerMe2021,BernardesPe2024_AM,CiGoPu2021,DAnUMa2021,Maiuriello2022}. In this paper we will focus on the particular idea of {\em chain recurrence}, which was originally introduced by Conley~\cite{Conley1972}, and then further developed by Conley~\cite{Conley1976,Conley1978} and Bowen~\cite{Bowen1975} together with the conception of {\em pseudotrajectory}.

It is worth mentioning that {\em chain recurrence} has intimate connections with the so-called Attractors and Stability Theory (see the textbooks \cite{AHi1994_book,Conley1978,Shub1987_book}) and that, from its origin, this notion has been deeply related to the study of differential equations: in the paper \cite{Conley1972} chain recurrence was related to the existence of progressive wave solutions for partial differential equations; in \cite{Conley1976} it was used to study some aspects of the qualitative theory of differential equations; and in \cite[Chapter~1]{Conley1978} it was a key tool to explore the stable properties of the solution set of ordinary differential equations. From the perspective of Linear Dynamics we have two possibilities:\\[-20pt]
\begin{enumerate}[--]
	\item we can relate this property to the theory of hypercyclicity (see for instance \cite[Theorem~A]{AntMantVar2022}, but also the results \cite[Theorems~7~and~12]{BernardesPe2024_AM}, where it is proved that {\em a continuous linear operator acting on a Banach space is topologically mixing, frequently hypercyclic and even densely distributionally chaotic as soon as the operator is \textbf{chain recurrent} and has the positive shadowing property});\\[-15pt]
	
	\item or we can just develop the theory regarding this specific ``{\em new}'' Linear Dynamics' property, by finding results and linear examples (usually on infinite-dimensional spaces) that can potentially behave in a very different way to the corresponding results from differentiable and compact dynamics.
\end{enumerate}
Here we will follow this second line of thought and, in particular, we completely solve in the negative the next problem posed by Nilson C.\ Bernardes Jr.\ and Alfred Peris in \cite{BernardesPe2024_AM}:

\begin{question}[\textbf{\cite[Problem~D]{BernardesPe2024_AM}}]\label{Q:main}
	Let $T:X\longrightarrow X$ be a continuous linear operator on a Banach space $X$ and let $CR(T)$ be its set of chain recurrent vectors. Is then $T\res_{CR(T)}$ chain recurrent?
\end{question}

The respective question has a positive answer for invertible compact dynamical systems. Indeed, the result \cite[Theorem~3.1.6]{AHi1994_book} shows that {\em given a homeomorphism $f:K\longrightarrow K$ for a compact metric space $K$, then the set of chain recurrent points for the restriction $f\res_{CR(f)}$ coincides with $CR(f)$, and hence $f\res_{CR(f)}$ is chain recurrent}. Our main result shows that this is no longer true for continuous linear operators, and completely solves Question~\ref{Q:main} in the negative, even on Hilbert spaces:

\begin{theorem}\label{The:main}
	There exist a Banach (even Hilbert) space $X$ and a continuous (and even invertible) linear operator $T:X\longrightarrow X$ such that the restriction of $T$ to its closed $T$-invariant subspace of chain recurrent vectors $T\res_{CR(T)}:CR(T)\longrightarrow CR(T)$ is not a chain recurrent operator. In particular, the operator $T$ can be constructed to fulfill that the subspace $CR(T)$ is infinite-dimensional while for its restriction $T\res_{CR(T)}$ we have the equality $CR(T\res_{CR(T)})=\{0_X\}$ for the zero-vector $0_X \in X$.
\end{theorem}

We would like to remark that, apart from solving Question~\ref{Q:main}, Theorem~\ref{The:main} also fights against a possible generalization of the very specific Hilbert-space result \cite[Corollary~2.10]{AntMantVar2022}: {\em the set of chain recurrent vectors of the restriction $T\res_{CR(T)}$ coincides with $CR(T)$ for every self-adjoint operator $T$ acting on a Hilbert space $H$}. In addition, the following was asked in \cite[Question~4.3]{AntMantVar2022}: {\em under which conditions on a $T$-invariant subspace $Y \subset X$, for an operator $T$ acting on a Banach space $X$, can we guarantee that $T\res_Y$ is chain recurrent?} In this case, Theorem~\ref{The:main} shows that the apparently natural condition ``$Y=CR(T)$'' is not enough, even if we assume that ``{\em $T$ is invertible}''. To prove Theorem~\ref{The:main} we use the class of {\em weighted backward shift operators on directed trees}.~In particular, the non-invertible case can be directly solved via relatively simple shifts on certain unrooted trees (see Theorem~\ref{The:non-invertible}) and, although the invertible case cannot be achieved via shifts, a slight modification of the previous counterexample will finish the work. In addition, we complete our study by showing that Question~\ref{Q:main} cannot be solved via classical shifts, therefore evidencing the great difference between {\em shifts on trees} and {\em classical shifts} in chain recurrence (see Theorem~\ref{The:classical} and Corollary~\ref{Cor:classical}).

The paper is organized as follows. In Section~\ref{Sec_2:notation} we introduce the notation and general background regarding {\em chain recurrence} and {\em weighted backward shift operators on directed trees}. In Section~\ref{Sec_3:non-invertible} we exhibit the non-invertible counterexample by using a weighted shift on a specific directed tree. Then, in Section~\ref{Sec_4:invertible}, we solve the invertible case by considering a slight modification of the example constructed in Section~\ref{Sec_3:non-invertible}. We finally show in Section~\ref{Sec_5:classical} that a classical (unilateral or bilateral) weighted backward shift acting on a Fr\'echet sequence space (over $\NN$ or $\ZZ$ respectively) is a chain recurrent operator as soon as it admits a non-zero chain recurrent vector.

\section{Notation and general background}\label{Sec_2:notation}

In some parts of \cite{BernardesPe2024_AM} the authors study {\em chain recurrence} in the context of linear maps on arbitrary, and not necessarily Hausdroff, topological vector spaces (they use the notion of {\em uniform space} and sometimes they even remove the continuity of the map), but since we will focus on operators acting on {\em Fr\'echet sequence spaces} we will just consider the following natural Linear Dynamics setting: from now on $T:X\longrightarrow X$ will be a {\em continuous linear operator} on a {\em real} or {\em complex Fr\'echet space} $X$, that is, a locally convex and completely metrizable topological vector space. We will briefly write $T \in \Lc(X)$ and the symbol $\KK$ will stand for either the real or complex field $\RR$ or $\CC$. Moreover, whenever we consider a Fr\'echet space $X$ we will assume that we have already chosen an increasing sequence of seminorms $(\|\cdot\|_k)_{k\in\NN}$ inducing its topology, which can also be endowed by the translation invariant metric $d(f,g):=\|f-g\|$ for each pair $(f,g) \in X\times X$, where the functional
\begin{equation}\label{eq:F-norm}
	\|f\| := \sum_{k\in\NN} 2^{-k} \cdot \min\left\{ 1 , \left\|f\right\|_k \right\} \quad \text{ for each } f \in X,
\end{equation}
is an {\em F-norm} for $X$ (see \cite[Section~2.1]{GrPe2011_book}). This allows us to treat the Fr\'echet case in a similar way to that in which $X$ is a Banach space, where the functional $\|\cdot\|$ will simply denote the norm of $X$.

\subsection{Chain recurrence in Linear Dynamics}

We can now define {\em chain recurrence} directly adapted to our Fr\'echet space setting:

\begin{definition}\label{Def:chain-rec}
	Given an operator $T \in \Lc(X)$ and a positive $\delta>0$, a finite sequence $(f_l)_{l=0}^m$ in $X$ is called a {\em $\delta$-chain for $T$ from $f_0$ to $f_m$}, also called {\em finite $\delta$-pseudotrajectory for $T$ from $f_0$ to $f_m$}, if
	\[
	\left\| f_l - T(f_{l-1}) \right\| < \delta \quad \text{ for every } 1\leq l\leq m,
	\]
	where the positive integer $m \in \NN$ is called the {\em length} of $(f_l)_{l=0}^m$. Moreover, a vector $f \in X$ is said to be {\em chain recurrent for $T$} if for every $\delta>0$ there exists a $\delta$-chain for $T$ from $f$ to itself. We will denote by $CR(T)$ the {\em set of chain recurrent vectors for $T$}, and the operator $T$ is said to be {\em chain recurrent} if the equality $CR(T)=X$ holds, that is, if every vector of the space is chain recurrent.
\end{definition}

Chain recurrence in Linear Dynamics was originally considered for operators on normed spaces in the 2022 paper \cite{AntMantVar2022}, although the chain recurrent behaviour for the particular case of classical unilateral and bilateral weighted backward shifts, acting on the also classical Banach $\ell^p$ and $c_0$ spaces, was already characterized in the 2021 paper \cite{AlBerMess2021}. In the linear context, the systematic study of chain recurrence and other similar notions depending on the concept of pseudotrajectory (such as {\em shadowing}) has been recently followed in \cite{BernardesPe2024_AM}. Along this paper we will use more than once that the set $CR(T)$ is always a closed $T$-invariant subspace for every continuous linear operator $T \in \Lc(X)$ as observed in the already mentioned works \cite[Corollary~2.3]{AntMantVar2022} and \cite[Proposition~26]{BernardesPe2024_AM}, so that the restriction $T\res_{CR(T)}$ considered in Question~\ref{Q:main} is again a continuous linear operator acting on the respective Banach space of chain recurrent vectors $CR(T)$. Moreover, the reader should keep in mind the following basic and immediate (but useful) remarks, that have been either proved, observed or even directly used in \cite{AntMantVar2022} and \cite{BernardesPe2024_AM}:

\begin{lemma}\label{Lem:useful}
	The following statements hold for every linear operator $T \in \Lc(X)$:
	\begin{enumerate}[{\em(a)}]
		\item If a finite sequence $(f_l)_{l=0}^m$ in $X$ is a $\delta$-chain for $T$, then there exists a second finite sequence $(g_l)_{l=1}^m$ in $X$ fulfilling that
		\[
		f_m = T^m(f_0) + \sum_{l=1}^m T^{m-l}(g_l) \quad \text{ and } \quad \|g_l\| < \delta \text{ for every } 1\leq l\leq m.
		\]
		
		\item A vector $f \in X$ is chain recurrent for $T$ if and only if for every $\delta>0$ there exists a $\delta$-chain for $T$ from $f$ to the zero-vector $0_X \in X$ and another $\delta$-chain for $T$ from the zero-vector $0_X \in X$ to $f$.
	\end{enumerate}
\end{lemma}
\begin{proof}
	Statement (a) follows by letting $g_l := f_l - T(f_{l-1})$ for each $1\leq l\leq m$. For the sufficiency in statement (b) note that we can concatenate $\delta$-chains finishing and starting at the zero-vector $0_X \in X$ because $T(0_X)=0_X$, while the necessity follows from the fact that a linear operator always has a unique chain recurrent class (see \cite[Proof of Theorem~2.1]{AntMantVar2022} and \cite[Lemma~8 and Proposition~25]{BernardesPe2024_AM}).
\end{proof}

\subsection{Fr\'echet sequence spaces and weighted shifts on directed trees}\label{SubSec_2.2:shifts}

Weighted shifts play a fundamental role in Operator Theory and Linear Dynamics: their very simple definition allows to explicitly compute the full orbit of every vector in the space while they remain a flexible class of operators where one can choose different weights to play with. This is why any new dynamical notion is usually first tested on shifts (as happened for chain recurrence in \cite{AlBerMess2021}), and they have provided many important counterexamples (see \cite[Section~4.1]{GrPe2011_book} and the references there). The generalized notion of {\em shifts acting on trees} has recently been investigated and here we will show that, for chain recurrence, they behave in a very different way than classical shifts. Following \cite{GrPa2023,GrPa2023_arXiv,JabJungStoc2012} a {\em directed tree} $(V,E)$, or just a {\em tree} $V$ for short, will be a connected directed graph consisting of a countable set of vertices $V$ and a set of directed edges $E \subset V\times V \setminus \{ (v,v) \ ; \ v \in V \}$ such that:\\[-17.5pt]
\begin{enumerate}[--]
	\item $V$ has no cycles;\\[-17.5pt]

	\item each vertex $v \in V$ has at most one {\em parent}, that is, a unique vertex $w \in V$ such that $(w,v) \in E$;\\[-17.5pt]

	\item there is at most one vertex with no parent, and in that case it is called the {\em root of $V$}.
\end{enumerate}
Given any vertex $v \in V$ its {\em parent} will be denoted by ``$\text{par}(v)$'', and the symbol ``$\text{Chi}(v)$'' will stand for the {\em set of all children} of $v \in V$, that is, the set formed by the vertices $u \in V$ such that $(v,u) \in E$. As an example one can think about the sets $\NN:=\{1,2,...\}$ or $\ZZ:=\{...,-2,-1,0,1,2,...\}$ with their usual tree structure, that is, for each integer $n \in \NN$ or $\ZZ$ we have that $\text{Chi}(n)=\{n+1\}$. In this case the set $\NN$ is a {\em tree with a root}, namely the positive integer $1 \in \NN$, while $\ZZ$ is an {\em unrooted tree}. We can now introduce the specific Fr\'echet spaces and linear operators that we will work with:

\begin{definition}\label{Def:shifts}
	For an arbitrary finite or countable set $V$ we will denote by $\KK^V$ the {\em Fr\'echet space of all (real or complex) sequences $f = (f(v))_{v\in V}$ over $V$} endowed with the usual product topology, which can be induced by the family of seminorms $(\|\cdot\|_k)_{k\in\NN}$, where
	\[
	\left\|f\right\|_k := \max_{v \in F_k} \left|f(v)\right| \quad \text{ for each } f = (f(v))_{v \in V} \in \KK^V,
	\]
	and where $(F_k)_{k\in\NN}$ is an increasing sequence of finite subsets of $V$ fulfilling that $V = \bigcup_{k\in\NN} F_k$. We will denote the {\em canonical unit sequences} of $\KK^V$ by $e_v := \chi_v$ for each $v \in V$, and a subspace $X \subset \KK^V$ is called a {\em Fr\'echet} (or {\em Banach}) {\em sequence space over $V$} if it is endowed with a Fr\'echet (resp.\ Banach) space topology for which the canonical embedding $X \hookrightarrow \KK^V$ is continuous.
	
	If now $V$ is a tree and we chose a {\em sequence of weights} denoted as $\sbf{\lambda} = (\lambda_v)_{v \in V} \in \KK^{V}$, the respective {\em weighted backward shift} $B_{\sbf{\lambda}}$ is formally defined on $\KK^V$ as
	\[
	[B_{\sbf{\lambda}}(f)](v) \ \ := \sum_{u \in \text{Chi}(v)} \lambda_u \cdot f(u) \quad \text{ for each } f = (f(v))_{v \in V} \in \KK^V \text{ and each } v \in V.
	\]
	By the Closed Graph Theorem, the restriction of $B_{\sbf{\lambda}}$ to a Fr\'echet sequence space $X$ over $V$ is continuous as soon as $B_{\sbf{\lambda}}:X\longrightarrow X$ is well-defined, that is, as soon as the shift $B_{\sbf{\lambda}}$ maps $X$ into itself. 
\end{definition}

Since Question~\ref{Q:main} was asked for operators acting on Banach spaces, in this paper we will focus on weighted backward shifts acting on the Banach sequence spaces
\[
\ell^p(V) := \left\{ f = (f(v))_{v\in V} \in \KK^V \ ; \ \sum_{v\in V} |f(v)|^p < \infty \right\} \quad \text{ for } 1\leq p<\infty,
\]
endowed with the norm $\|f\|_p := \left( \sum_{v\in V} |f(v)|^p \right)^{1/p}$ for each $f = (f(v))_{v\in V} \in \ell^p(V)$; but we will also consider the associated $c_0$-space
\[
c_0(V) := \left\{ f = (f(v))_{v\in V} \in \KK^V \ ; \ \forall \eps>0, \exists F_{\eps} \subset V \text{ finite}, \forall v \in V\setminus F_{\eps}, |f(v)|<\eps \right\},
\]
endowed with the norm $\|f\|_{\infty} := \sup_{v\in V} |f(v)|$ for each $f = (f(v))_{v\in V} \in c_0(V)$. Note that, when we have $V=\NN$ or $\ZZ$ with their usual tree structure (that is, $\text{Chi}(n)=\{n+1\}$ for each $n \in V$), then the respective shifts coincide with the classical definition of these operators, that is, with
\[
(f(n))_{n\in V} \longmapsto (\lambda_{n+1} \cdot f(n+1))_{n\in V},
\]
for any fixed sequence of weights $\sbf{\lambda}=(\lambda_n)_{n\in V}$ in $\KK^{\NN}$ or $\KK^{\ZZ}$ (see Section~\ref{Sec_5:classical} for more on classical unilateral and bilateral backward shifts). We are now ready to construct our non-invertible counterexample.

\section{A bilateral backward shift with some extra branches}\label{Sec_3:non-invertible}

In this section we prove the non-invertible case of Theorem~\ref{The:main} by considering a weighted backward shift acting on a directed tree, which solves Question~\ref{Q:main} negatively in a relatively simple way. In particular, we prove the following weakened version of Theorem~\ref{The:main}:

\begin{theorem}\label{The:non-invertible}
	There exists a weighted backward shift $B_{\sbf{\lambda}}:X\longrightarrow X$ acting on a directed tree Banach (even Hilbert) sequence space $X$ such that the restriction of $B_{\sbf{\lambda}}$ to its closed $B_{\sbf{\lambda}}$-invariant subspace of chain recurrent vectors $B_{\sbf{\lambda}}\res_{CR(B_{\sbf{\lambda}})}:CR(B_{\sbf{\lambda}})\longrightarrow CR(B_{\sbf{\lambda}})$ is not a chain recurrent operator. In particular, $B_{\sbf{\lambda}}$ can be constructed to fulfill that the subspace $CR(B_{\sbf{\lambda}})$ is infinite-dimensional while for its restriction $B_{\sbf{\lambda}}\res_{CR(B_{\sbf{\lambda}})}$ we have the equality $CR(B_{\sbf{\lambda}}\res_{CR(B_{\sbf{\lambda}})})=\{0_X\}$ for the zero-vector $0_X \in X$.
\end{theorem}

\textbf{\textit{The rest of this section is devoted to prove Theorem~\ref{The:non-invertible}}}:

Let $V := \ZZ \cup \{ (-k,j) \ ; \ k \in \NN \text{ and } 1\leq j\leq k \}$. We consider $V$ as a tree letting
\[
\begin{cases}
	\text{par}(n) := n-1 & \text{ for each } n \in \ZZ,\\[7.5pt]
	\text{par}((-k,j)) := (-k,j-1) & \text{ for each } k \in \NN \text{ and } 1< j\leq k,\\[7.5pt]
	\text{par}((-k,1)) := -k & \text{ for each } k \in \NN.
\end{cases}
\]
See Figure~\ref{Fig:tree} for a graphic representation of the respective directed tree $V$.\newpage
\begin{figure}[H]
	\begin{center}
		\includegraphics[width=16cm]{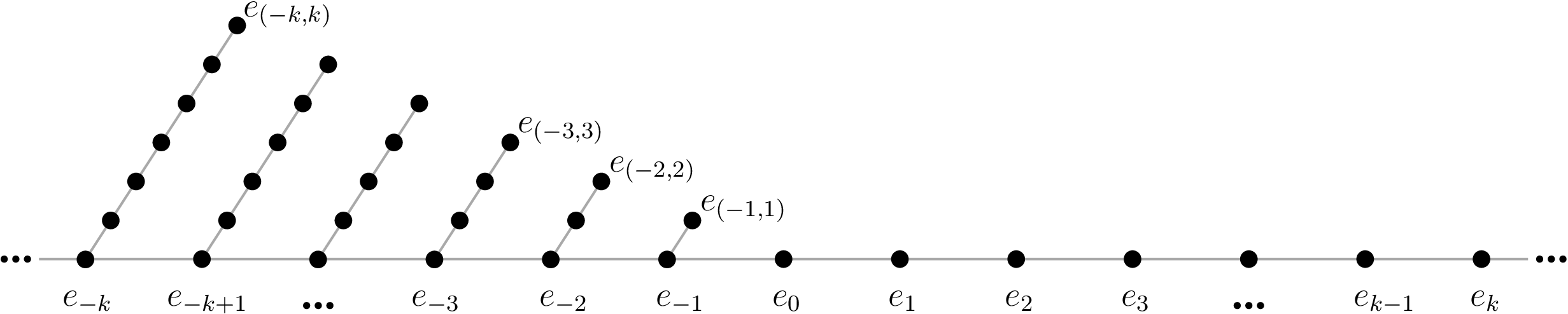}
		\caption{Graphic representation of the directed tree $V$ from Section~\ref{Sec_3:non-invertible}.}\label{Fig:tree}
	\end{center}
\end{figure}

We now fix two (real or complex) values $\mu_1,\mu_2 \in \KK$ fulfilling that $1<|\mu_1|<|\mu_2|$, and we consider the sequence of weights $\sbf{\lambda}=(\lambda_v)_{v \in V} \in \KK^V$ as
\[
\begin{cases}
	\lambda_n := \mu_1 & \text{ for each } n \in \ZZ,\\[7.5pt]
	\lambda_{(-k,j)} := \mu_2 & \text{ for each } k \in \NN \text{ and } 1\leq j\leq k.
\end{cases}
\]
Hence, the image under the associated weighted backward shift $B_{\sbf{\lambda}}$ for each of the canonical unit sequences $(e_v)_{v\in V}$ is the following
\[
\begin{cases}
	e_n \longmapsto \mu_1 \cdot e_{n-1} & \text{ for each } n \in \ZZ,\\[7.5pt]
	e_{(-k,j)} \longmapsto \mu_2 \cdot e_{(-k,j-1)} & \text{ for each } k \in \NN \text{ and } 1< j\leq k,\\[7.5pt]
	e_{(-k,1)} \longmapsto \mu_2 \cdot e_{-k} & \text{ for each } k \in \NN.
\end{cases}
\]
See Figure~\ref{Fig:shift} for a graphic representation of the respective backward shift $B_{\sbf{\lambda}}$.
\begin{figure}[H]
	\begin{center}
		\includegraphics[width=16cm]{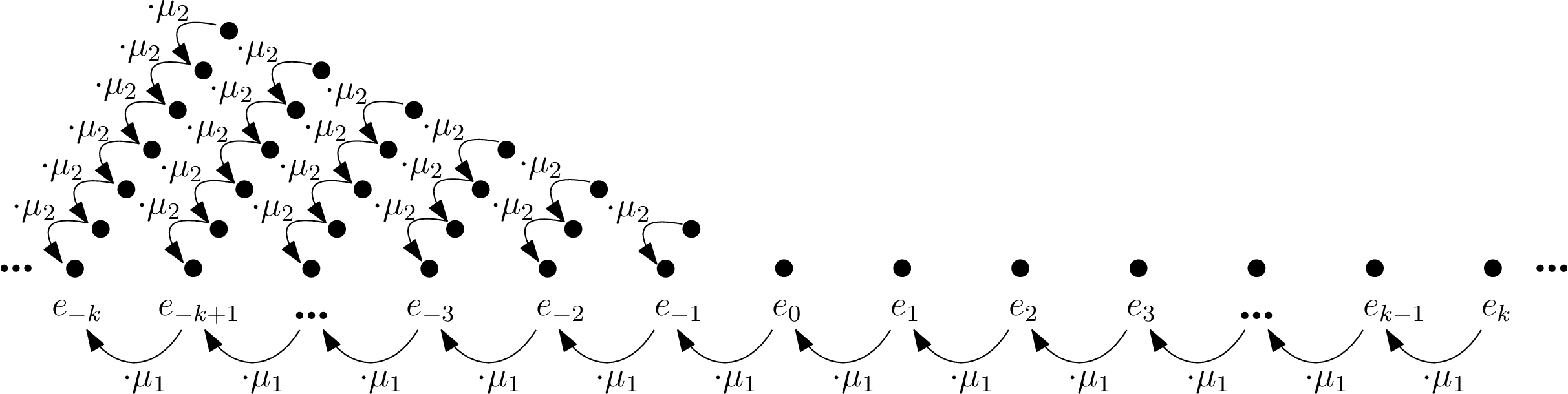}
		\caption{Graphic representation of the backward shift $B_{\sbf{\lambda}}$ from Section~\ref{Sec_3:non-invertible}.}\label{Fig:shift}
	\end{center}
\end{figure}

From now on we fix $X = \ell^p(V)$, with $1\leq p<\infty$, or $X=c_0(V)$ endowed with their usual norms, explicitly stated in Subsection~\ref{SubSec_2.2:shifts}. Note that the previous backward shift $B_{\sbf{\lambda}}:X\longrightarrow X$ is continuous because the sequence of weights $\sbf{\lambda}=(\lambda_v)_{v \in V}$ is bounded. Moreover, since $\ell^2(V)$ is a Hilbert space, the proof of Theorem~\ref{The:non-invertible} will be complete as soon as we check the following equality:
\begin{equation}\label{eq:main}
	CR(B_{\sbf{\lambda}}) = \cl{\lspan\{ e_n \ ; \ n \in \ZZ \}}.
\end{equation}
Indeed, if \eqref{eq:main} holds then $B_{\sbf{\lambda}}\res_{CR(B_{\sbf{\lambda}})}$ is exactly the classical bilateral backward shift on $\ZZ$ but multiplied by the parameter $\mu_1$, and since $|\mu_1|>1$ we have that the operator $B_{\sbf{\lambda}}\res_{CR(B_{\sbf{\lambda}})}$ is a {\em proper dilation} as defined in \cite[Page~6]{AntMantVar2022}, so that its set of chain recurrent vectors is exactly the singleton formed by the zero-vector $0_X$ (see \cite[Corollary~2.7]{AntMantVar2022} but also Theorem~\ref{The:classical} below). Let us check that \eqref{eq:main} holds.

\subsection[Calculating CR(B): the first inclusion]{Calculating $CR(B_{\sbf{\lambda}})$: the first inclusion}\label{SubSec_3.1:first}

We start by proving that
\begin{equation}\label{eq:inclusion1}
	CR(B_{\sbf{\lambda}}) \supset \cl{\lspan\{ e_n \ ; \ n \in \ZZ \}}.
\end{equation}
Since $CR(B_{\sbf{\lambda}})$ is a closed subspace (see \cite[Corollary~2.3]{AntMantVar2022} or \cite[Proposition~26]{BernardesPe2024_AM}) it is enough to check that the vector $e_n$ belongs to $CR(B_{\sbf{\lambda}})$ for every $n \in \ZZ$. We start with $e_0$:

\begin{fact}\label{Fact:e0.CR}
	The canonical unit sequence $e_0$ belongs to $CR(B_{\sbf{\lambda}})$.
\end{fact}
\begin{proof}
	Fix any $\delta>0$. By Lemma~\ref{Lem:useful} it is enough to find a $\delta$-chain for $B_{\sbf{\lambda}}$ from the vector $e_0$ to the zero-vector $0_X \in X$ and another $\delta$-chain for $B_{\sbf{\lambda}}$ from $0_X$ to $e_0$. Let us start by this last chain:
	\begin{enumerate}[--]
		\item \textbf{Step 1: from $0_X$ to  $e_0$}. Find $m_1 \in \NN$ with $m_1>1$ and such that $1 < \delta \cdot |\mu_1|^{m_1-1}$, and consider the chain
		\[
		\begin{cases}
			f_0 := 0_X, &\\[7.5pt]
			f_1 := \left( \tfrac{1}{\mu_1} \right)^{m_1-1} \cdot \ e_{m_1-1}, &\\[12.5pt]
			f_l := B_{\sbf{\lambda}}^{l-1}(f_1) & \text{ for } l=2,...,m_1.
		\end{cases}
		\]
		This is a $\delta$-chain since $\| f_1-B_{\sbf{\lambda}}(f_0) \| \ = \ \left( \tfrac{1}{|\mu_1|} \right)^{m_1-1} < \ \delta$. Moreover, $f_0=0_X$ and
		\[
		f_{m_1} = B_{\sbf{\lambda}}^{m_1-1}(f_1) = \mu_1^{m_1-1} \cdot \left( \tfrac{1}{\mu_1} \right)^{m_1-1} \cdot \ e_0 = e_0.
		\]
		
		\item \textbf{Step 2: from $e_0$ to $0_X$}. For this case find $m_2 \in \NN$ with $m_2>1$ and such that $|\mu_1|^{m_2} < \delta \cdot |\mu_2|^{m_2-1}$, and consider the chain
		\[
		\begin{cases}
			f_0 := e_0, &\\[7.5pt]
			f_1 := B_{\sbf{\lambda}}(f_0) - \left( \tfrac{1}{\mu_2} \right)^{m_2-1} \cdot \ \mu_1^{m_2} \ \cdot \ e_{(-m_2,m_2-1)}, &\\[12.5pt]
			f_l := B_{\sbf{\lambda}}^{l-1}(f_1) & \text{ for } l=2,...,m_2.
		\end{cases}
		\]
		This is a $\delta$-chain since $\| f_1-B_{\sbf{\lambda}}(f_0) \| = \left( \tfrac{1}{|\mu_2|} \right)^{m_2-1} \cdot \ |\mu_1|^{m_2} < \delta$. Moreover, $f_0=e_0$ and
		\begin{align*}
			f_{m_2} &= B_{\sbf{\lambda}}^{m_2-1}(f_1) = B_{\sbf{\lambda}}^{m_2}(f_0) - B_{\sbf{\lambda}}^{m_2-1}\left( \left(\tfrac{1}{\mu_2}\right)^{m_2-1} \cdot \ \mu_1^{m_2} \ \cdot \ e_{(-m_2,m_2-1)}\right)\\
			&= \mu_1^{m_2} \cdot e_{-m_2} - \mu_2^{m_2-1} \cdot \left(\tfrac{1}{\mu_2}\right)^{m_2-1} \cdot \ \mu_1^{m_2} \cdot e_{-m_2} = \mu_1^{m_2} \cdot e_{-m_2} - \mu_1^{m_2} \cdot e_{-m_2} = 0_X.
		\end{align*}
	\end{enumerate}
	The arbitrariness of $\delta>0$ together with \textbf{Step~1} and \textbf{Step~2} show that $e_0 \in CR(B_{\sbf{\lambda}})$.
\end{proof}

Finally, given any positive integer $n \in \NN$, it is not hard to repeat similar arguments to those employed in Fact~\ref{Fact:e0.CR} to show that $e_{-n} \in CR(B_{\sbf{\lambda}})$ since there are enough extra branches along the left side of the underlying tree $V$ to repeat the same relations stated but ``shifted $-n$ positions''. For the vector $e_n$ we can construct a $\delta$-chain for $B_{\sbf{\lambda}}$ from $0_X$ to $e_n$ as in \textbf{Step~1} of Fact~\ref{Fact:e0.CR}, while for the required $\delta$-chain from $e_n$ to $0_X$ we can use the equality $B_{\sbf{\lambda}}^n(e_n) = \mu_1^n \cdot e_0$ and from there to slightly modify \textbf{Step~2} of Fact~\ref{Fact:e0.CR} by considering a longer chain if necessary. We deduce that \eqref{eq:inclusion1} holds.

\subsection[Calculating CR(B): the second inclusion]{Calculating $CR(B_{\sbf{\lambda}})$: the second inclusion}\label{SubSec_3.2:second}

We now complete the proof of Theorem~\ref{The:non-invertible} by showing that
\begin{equation}\label{eq:inclusion2}
	CR(B_{\sbf{\lambda}}) \subset \cl{\lspan\{ e_n \ ; \ n \in \ZZ \}}.
\end{equation}
In particular, we will have that \eqref{eq:inclusion2} holds as soon as we prove the following fact:

\begin{fact}\label{Fact:non.CR}
	Let $f = (f(v))_{v\in V} \in X$. If $f \notin \cl{\lspan\{ e_n \ ; \ n \in \ZZ \}}$, then $f \notin CR(B_{\sbf{\lambda}})$.
\end{fact}
\begin{proof}
	By assumption, there exists $k \in \NN$ such that $|f(-k,j)|>0$ for some $1\leq j\leq k$. Let
	\[
	j_k := \max\{ 1 \leq j \leq k \ ; \ f(-k,j)\neq 0 \},
	\]
	and consider a positive value $\delta>0$ small enough to fulfill the inequality
	\begin{equation}\label{eq:delta}
		\delta < \frac{\left| f(-k,j_k) \right|}{(k-j_k+1) \cdot |\mu_2|^{(k-j_k)}}.
	\end{equation}
	By contradiction assume that $f \in CR(B_{\sbf{\lambda}})$, and hence that there exists a $\delta$-chain $(f_l)_{l=0}^m$ for $B_{\sbf{\lambda}}$ from the vector $f$ to itself. Thus, by Lemma~\ref{Lem:useful} there exists a finite sequence $(g_l)_{l=1}^m$ in $X$ such that
	\begin{equation}\label{eq:coordinate}
		f = f_m = B_{\sbf{\lambda}}^m(f_0) + \sum_{l=1}^m B_{\sbf{\lambda}}^{m-l}(g_l) \quad \text{ and } \quad | g_l(v) | < \delta \text{ for all } 1\leq l\leq m \text{ and } v \in V,
	\end{equation}
	which implies that
	\begin{equation}\label{eq:f}
		f(-k,j_k) = [B_{\sbf{\lambda}}^m(f)](-k,j_k) + \sum_{l=0}^{m-1} [B_{\sbf{\lambda}}^{l}(g_{m-l})](-k,j_k).
	\end{equation}
	From now on we have two possibilities:
	\begin{enumerate}[--]
		\item \textbf{Case 1: $m \leq k-j_k$}. Hence, by \eqref{eq:f} and since $1\leq m\leq k-j_k$ we have that
		\[
		f(-k,j_k) = \mu_2^m \cdot f(-k,j_k+m) + \sum_{l=0}^{m-1} \mu_2^l \cdot g_{m-l}(-k,j_k+l).
		\]
		Moreover, $f(-k,j_k+m)=0$ by the definition of $j_k$ so that using \eqref{eq:coordinate} and then \eqref{eq:delta} we get that
		\[
		\left| f(-k,j_k) \right| \leq \sum_{l=0}^{m-1} |\mu_2|^l \cdot \left| g_{m-l}(-k,j_k+l) \right| < (k-j_k) \cdot |\mu_2|^{(k-j_k)} \cdot \delta < \left| f(-k,j_k) \right|,
		\]
		contradicting the existence of the $\delta$-chain $(f_l)_{l=0}^m$.
		
		\item \textbf{Case 2: $m > k-j_k$}. In this case we start by noticing that $[B_{\sbf{\lambda}}^l(g)](-k,j_k)=0$ for every $g \in X$ and every $l>k-j_k$. Hence, by \eqref{eq:f} we have the equality
		\[
		f(-k,j_k) = \sum_{l=0}^{k-j_k} [B_{\sbf{\lambda}}^l(g_{m-l})](-k,j_k) = \sum_{l=0}^{k-j_k} \mu_2^l \cdot g_{m-l}(-k,j_k+l),
		\]
		so that, using again \eqref{eq:coordinate} and then \eqref{eq:delta}, we get the contradiction
		\[
		\left| f(-k,j_k) \right| \leq \sum_{l=0}^{k-j_k} |\mu_2|^l \cdot \left| g_{m-l}(-k,j_k+l) \right| < (k-j_k+1) \cdot |\mu_2|^{(k-j_k)} \cdot \delta < \left| f(-k,j_k) \right|,
		\]
		which finally shows that $f \notin CR(B_{\sbf{\lambda}})$.\qedhere
	\end{enumerate}
\end{proof}

\subsection{The subspace of chain recurrent vectors can be finite-dimensional}\label{SubSec_3.3:finite}

The arguments exhibited in Subsections~\ref{SubSec_3.1:first} and \ref{SubSec_3.2:second} complete the proof of Theorem~\ref{The:non-invertible} in its full generality. In fact, the set of chain recurrent vectors obtained $CR(B_{\sbf{\lambda}})$ is an infinite-dimensional subspace of the respective Banach space $X$ while $CR(B_{\sbf{\lambda}}\res_{CR(B_{\sbf{\lambda}})})=\{0_X\}$. However, the reader may ask the following question proposed to us by the anonymous reviewer:
\begin{enumerate}[--]
	\item {\em Is there any operator $T \in \Lc(X)$ acting on a Banach or Fr\'echet space $X$ for which $CR(T)\neq\{0_X\}$ is a finite-dimensional subspace of $X$ such that $CR(T\res_{CR(T)})=\{0_X\}$?}
\end{enumerate}

The answer is positive and one can also achieve this behaviour with a backward shift on a directed tree. For instance, let $V' := \left(\ZZ\setminus\NN\right) \cup \{ (-k,j) \ ; \ k \in \NN \text{ and } 1\leq j\leq k \}$, and consider it as a tree with
\[
\begin{cases}
	\text{par}(n) := n-1 & \text{ for each } n \in \ZZ\setminus\NN,\\[7.5pt]
	\text{par}((-k,j)) := (-k,j-1) & \text{ for each } k \in \NN \text{ and } 1< j\leq k,\\[7.5pt]
	\text{par}((-k,1)) := 0 & \text{ for all } k \in \NN.
\end{cases}
\]
See Figure~\ref{Fig:tree2} for a graphic representation of the respective directed tree $V'$.
\begin{figure}[H]
	\begin{center}
		\includegraphics[width=16cm]{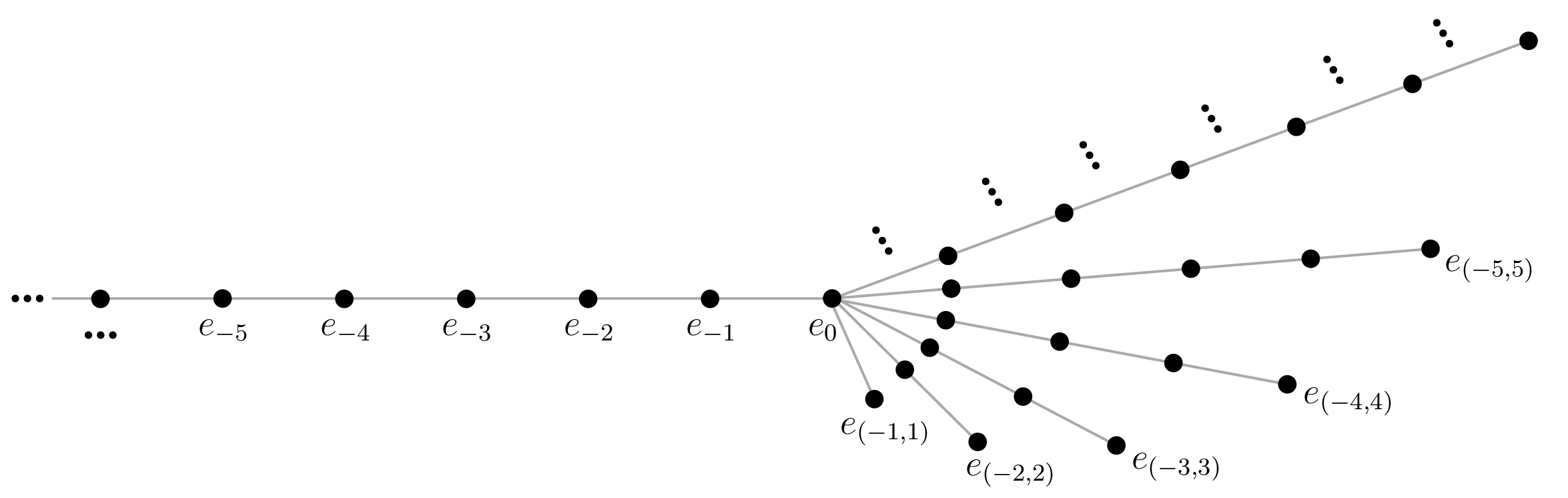}
		\caption{Graphic representation of the directed tree $V'$ from Subsection~\ref{SubSec_3.3:finite}.}\label{Fig:tree2}
	\end{center}
\end{figure}

From now on we fix $X = \ell^p(V')$, with $1\leq p<\infty$, or $X=c_0(V')$ endowed with their usual norms explicitly stated in Subsection~\ref{SubSec_2.2:shifts}. We are going to consider a continuous (or well-defined) backward shift $B_{\sbf{\lambda'}}:X\longrightarrow X$ associated to a sequence of weights $\sbf{\lambda'}=(\lambda_v')_{v \in V'} \in \KK^{V'}$ fulfilling that $CR(B_{\sbf{\lambda'}})$ is a finite-dimensional subspace while $CR(B_{\sbf{\lambda'}}\res_{CR(B_{\sbf{\lambda'}})})=\{0_X\}$. To this end we fix some $N \in \NN$, which will be the dimension of $CR(B_{\sbf{\lambda'}})$, and given a real or complex value $\mu \in \KK$ fulfilling that $|\mu|>2$ we can consider the sequence of weights $\sbf{\lambda'}=(\lambda_v')_{v \in V'} \in \KK^{V'}$ with
\[
\begin{cases}
	\lambda_n' := 0 & \text{ for each } n \in \ZZ\setminus\NN \text{ with } |n|\geq N-1,\\[7.5pt]
	\lambda_n' := 1 & \text{ for each } n \in \ZZ\setminus\NN \text{ with } |n|<N-1,\\[7.5pt]
	\lambda_{(-k,j)}' := \mu & \text{ for each } k \in \NN \text{ and } 1< j\leq k,\\[7.5pt]
	\lambda_{(-k,1)}' := \tfrac{1}{2^k} & \text{ for each } k \in \NN.
\end{cases}
\]
Hence, the image under the associated weighted backward shift $B_{\sbf{\lambda'}}$ for each of the canonical unit sequences $(e_v)_{v\in V}$ is the following
\[
\begin{cases}
	e_n \longmapsto 0_X & \text{ for each } n \in \ZZ\setminus\NN \text{ with } |n|\geq N-1,\\[7.5pt]
	e_n \longmapsto e_{n-1} & \text{ for each } n \in \ZZ\setminus\NN \text{ with } |n|<N-1,\\[7.5pt]
	e_{(-k,j)} \longmapsto \mu \cdot e_{(-k,j-1)} & \text{ for each } k \in \NN \text{ and } 1< j\leq k,\\[7.5pt]
	e_{(-k,1)} \longmapsto \tfrac{1}{2^k} \cdot e_0 & \text{ for each } k \in \NN.
\end{cases}
\]
See Figure~\ref{Fig:shift2} for a graphic representation of the respective backward shift $B_{\sbf{\lambda'}}$.
\begin{figure}[H]
	\begin{center}
		\includegraphics[width=16cm]{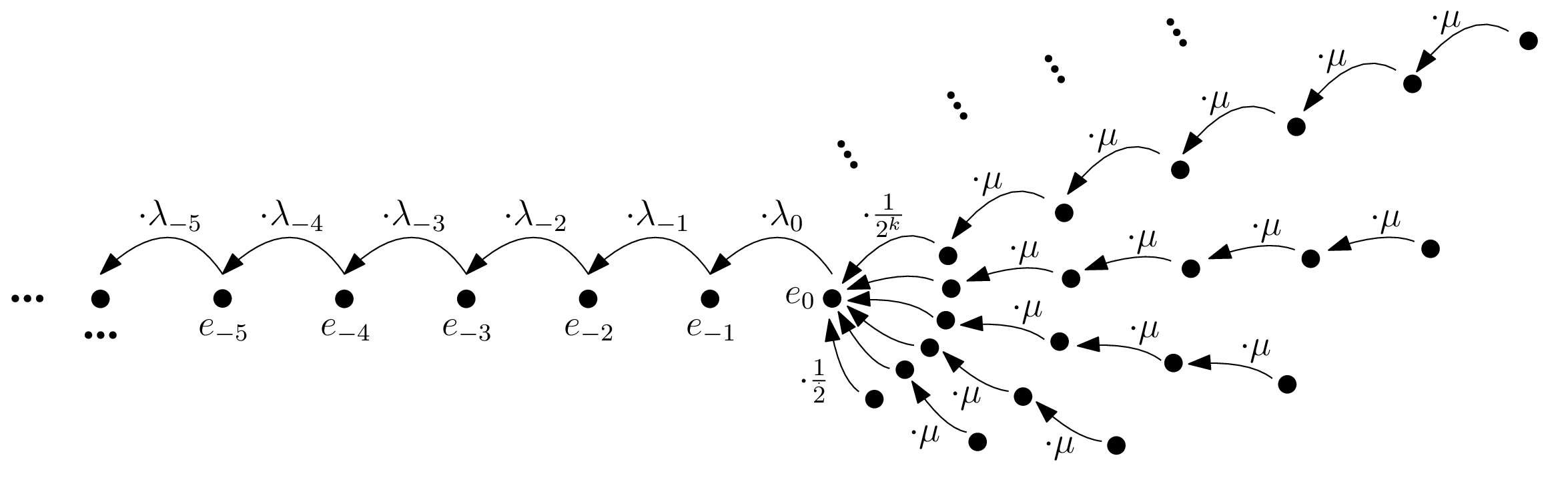}
		\caption{Graphic representation of the backward shift $B_{\sbf{\lambda'}}$ from Subsection~\ref{SubSec_3.3:finite}.}\label{Fig:shift2}
	\end{center}
\end{figure}

Using the Closed Graph Theorem, one can easily check that $B_{\sbf{\lambda'}}:X\longrightarrow X$ is continuous since:
\begin{enumerate}[--]
	\item the sequence of weights $\sbf{\lambda'}=(\lambda_v')_{v \in V'} \in \KK^{V'}$ is bounded;
	
	\item and although $0 \in V'$ is a vertex with infinitely many children, $\text{Chi}(0)=\{ (-k,1) \ ; \ k \in \NN \}$, the weights associated to these children form an $\ell^1(\NN)$-sequence, namely $(\tfrac{1}{2^k})_{k\in\NN}$.
\end{enumerate}
Similar backward shifts have been used in \cite{MenetPa2024_arXiv} and our shift $B_{\sbf{\lambda'}}$ is in the spirit of the operator considered in \cite[Proposition~3.10]{MenetPa2024_arXiv}. Arguing as in Subsections~\ref{SubSec_3.1:first} and \ref{SubSec_3.2:second} one can check that
\begin{equation}\label{eq:main2}
	CR(B_{\sbf{\lambda'}}) = \cl{\lspan\{ e_n \ ; \ n \in \ZZ\setminus\NN \text{ with } |n|<N \}},
\end{equation}
which fulfills that $\dim(CR(B_{\sbf{\lambda'}}))=N$. Moreover, once \eqref{eq:main2} holds, it is not hard to check that the restricted operator $B_{\sbf{\lambda'}}\res_{CR(B_{\sbf{\lambda'}})}:CR(B_{\sbf{\lambda'}})\longrightarrow CR(B_{\sbf{\lambda'}})$ is a finite-dimensional operator (a matrix) whose spectrum is the singleton formed by the value $0$. Hence, such a restriction is a {\em proper contraction} as defined in \cite[Page~6]{AntMantVar2022}, so that one can finally deduce the desired equality $CR(B_{\sbf{\lambda'}}\res_{CR(B_{\sbf{\lambda'}})})=\{0_X\}$. See \cite[Proposition~2.5]{AntMantVar2022}, but also Lemma~\ref{Lem:finite-dimension} and Theorem~\ref{The:finite-dimension} below, for more details.

Let us mention that an easy modification of the backward shift considered in the previous lines can be used to obtain, for any pair of positive integers $1\leq L<M \in \NN$, a continuous linear operator $T \in \Lc(X)$ fulfilling that $\dim(CR(T))=M$ and $\dim(CR(T\res_{CR(T)})) = L$. Indeed, if apart from the tree $V'$ and the backward shift $B_{\sbf{\lambda'}}$ considered in Figures~\ref{Fig:tree2}~and~\ref{Fig:shift2} we let $M=N+L$ and we add some extra vertices (say $\{1,2,...,L\}$), then the operator that equals $B_{\sbf{\lambda'}}$ in $V'$ and leaves the extra canonical unit sequences $e_1=\chi_1$, $e_2=\chi_2$, ..., $e_L=\chi_L$ as fixed points fulfills the required conditions.

\section{The invertible case}\label{Sec_4:invertible}

In this section we construct an invertible counterexample for Question~\ref{Q:main}, hence completing the proof of Theorem~\ref{The:main}. This example is highly inspired by the weighted shift considered in Section~\ref{Sec_3:non-invertible}, but in this case we are going to explore a bit more the possible parameters that we can choose to construct such an invertible operator. Our counterexample in this section can be seen as a backward shift on a directed graph, instead than on a directed tree, as noticed by the anonymous reviewer. This class of operators has been considered in \cite{BaranovLiPa2023_arXiv}, and every continuous linear operator on a Banach space with a Schauder basis can be seen as a backward shift on a directed graph (see \cite[Remark~7.3]{BaranovLiPa2023_arXiv}).

Before starting the construction we would like to remark that weighted shifts on directed trees cannot be used to achieve an invertible counterexample for Question~\ref{Q:main}. Indeed, if we consider a directed tree $(V,E)$ with a vertex $v \in V$ having more than one children (that is, if $(v,u)$ and $(v,w)$ belong to the set of directed edges $E$ for two different vertices $u\neq w \in V$), then any weighted backward shift on $V$ is not an injective operator. This trivial observation shows that the only invertible shifts acting on directed trees (as defined in Definition~\ref{Def:shifts}) are the classical weighted backward shifts on the unrooted tree $\ZZ$ and, as we show in Section~\ref{Sec_5:classical}, one cannot solve Question~\ref{Q:main} with these operators because a classical (unilateral or bilateral) weighted backward shift is a chain recurrent operator as soon as it admits a non-zero chain recurrent vector (see Theorem~\ref{The:classical} and Corollary~\ref{Cor:classical} below).

Moreover, note that the particular operator exhibited in Section~\ref{Sec_3:non-invertible} was also not surjective since every vector of the type $e_{(-k,k)}$, for each $k \in \NN$, had no preimage. We can solve both the injectivity and surjectivity problems by adding some extra points in the starting set $V$.

\textbf{\textit{The rest of this section is devoted to prove the invertible case of Theorem~\ref{The:main}}}:

From now on consider the countable set
\[
V := \ZZ \cup \left( \bigcup_{k\in\NN} \{ (-k,j) \ ; \ j \in \ZZ \} \right),
\]
and the respective space $\KK^V$ of all (real or complex) sequences over $V$. As in Section~\ref{Sec_3:non-invertible} we now fix two values $\mu_1,\mu_2 \in \KK$ fulfilling that $1<|\mu_1|<|\mu_2|$. Now we consider a kind of infinite matrix, indexed by the countable set $\NN\times\ZZ$ and denoted by $(\lambda_{(-k,j)})_{(k,j)\in\NN\times\ZZ} \in \KK^{\NN\times\ZZ}$, for which
\[
\lambda_{(-k,j)} :=
\begin{cases}
	\mu_2 & \text{ for each } k \in \NN \text{ and } \hspace{0.45cm} 1\leq j\leq k,\\[7.5pt]
	\frac{1}{\mu_2} & \text{ for each } k \in \NN \text{ and } -k\leq j\leq 0,
\end{cases}
\]
but also fulfilling that
\begin{equation}\label{eq:continuity}
	0 \ < \ \inf_{(k,j) \in \NN\times\ZZ} \left|\lambda_{(-k,j)}\right| \ \leq \ \sup_{(k,j) \in \NN\times\ZZ} \left|\lambda_{(-k,j)}\right| \ < \ \infty,
\end{equation}
and at least one of the following conditions
\begin{equation}\label{eq:non.CR}
	\sum_{j=1}^{\infty} \left| \lambda_{(-k,1)} \cdots \lambda_{(-k,j)} \right| < \infty \quad \text{ or } \quad \sum_{j=1}^{\infty} \left| \lambda_{(-k,-(j-1))} \cdots \lambda_{(-k,0)} \right|^{-1} < \infty \quad \text{ for each } k \in \NN.
\end{equation}
Note that \eqref{eq:continuity} and both conditions stated in \eqref{eq:non.CR} can be achieved for every $k \in \NN$ at the same time if, for instance, we let
\[
\lambda_{(-k,j)} =
\begin{cases}
	\frac{1}{\mu_2} & \text{ for each pair } (k,j) \in \NN\times\ZZ \text{ with } \ k<j,\\[7.5pt]
	\mu_2 & \text{ for each pair } (k,j) \in \NN\times\ZZ \text{ with } \ j<-k,
\end{cases}
\]
so that there exist many matrices $(\lambda_{(-k,j)})_{(k,j)\in\NN\times\ZZ} \in \KK^{\NN\times\ZZ}$ fulfilling the required conditions.\newpage

From now on let $T:\KK^V\longrightarrow\KK^V$ be the unique linear map that acts on each of the canonical unit sequences $e_v = \chi_v$, for $v \in V$, as
\[
T(e_v) :=
\begin{cases}
	\mu_1 \cdot e_{n-1} & \text{ if } v=n \in \ZZ,\\[7.5pt]
	\lambda_{(-k,j)} \cdot e_{(-k,j-1)} & \text{ if } v=(-k,j) \text{ with } k \in \NN \text{ and } j \in \ZZ\setminus\{1\},\\[7.5pt]
	\mu_2 \cdot \left( e_{(-k,0)} + e_{-k} \right) & \text{ if } v=(-k,1) \text{ for some } k \in \NN.
\end{cases}
\]
As in Section~\ref{Sec_3:non-invertible} let $X$ be any of the Banach (or Hilbert) spaces $\ell^p(V)$, with $1\leq p<\infty$, or $c_0(V)$ endowed with their usual norms (see Subsection~\ref{SubSec_2.2:shifts}). Note that the restriction of $T$ to the respective sequence space $X$, still denoted by $T:X\longrightarrow X$, is continuous and even invertible by \eqref{eq:continuity}. Indeed, it is not hard to check that $T^{-1}:X\longrightarrow X$ is the continuous operator fulfilling that
\[
T^{-1}(e_v) =
\begin{cases}
	\frac{1}{\mu_1} \cdot e_{n+1} & \text{ if } v=n \in \ZZ,\\[7.5pt]
	\frac{1}{\lambda_{(-k,j+1)}} \cdot e_{(-k,j+1)} & \text{ if } v=(-k,j) \text{ with } k \in \NN \text{ and } j \in \ZZ\setminus\{0\},\\[7.5pt]
	\frac{1}{\mu_2} \cdot e_{(-k,1)} - \frac{1}{\mu_1} \cdot e_{-k+1} & \text{ if } v=(-k,0) \text{ for some } k \in \NN.
\end{cases}
\]
Moreover, and exactly as in Section~\ref{Sec_3:non-invertible}, the proof of Theorem~\ref{The:main} will be complete as soon as we check the following equality:
\begin{equation}\label{eq:mainT}
	CR(T) = \cl{\lspan\{ e_n \ ; \ n \in \ZZ \}}.
\end{equation}
Indeed, if \eqref{eq:mainT} holds then $T\res_{CR(T)}$ is again the classical bilateral backward shift on $\ZZ$ multiplied by the parameter $\mu_1$, and again $|\mu_1|>1$ so that $T\res_{CR(T)}$ is a {\em proper dilation} as defined in \cite[Page~6]{AntMantVar2022} and its set of chain recurrent vectors is precisely the singleton formed by the zero-vector $0_X \in X$ as shown in \cite[Corollary~2.7]{AntMantVar2022}; see also Theorem~\ref{The:classical} below. Let us check that \eqref{eq:mainT} holds.

\subsection[Calculating CR(T): the first inclusion]{Calculating $CR(T)$: the first inclusion}

Following the strategy exhibited in Section~\ref{Sec_3:non-invertible}, we start by proving that
\begin{equation}\label{eq:inclusion1T}
	CR(T) \supset \cl{\lspan\{ e_n \ ; \ n \in \ZZ \}}.
\end{equation}
Recall that it is enough checking $e_n \in CR(T)$ for every $n \in \ZZ$. As we argued in Section~\ref{Sec_3:non-invertible}:

\begin{fact}\label{Fact:e0.CR.T}
	The canonical unit sequence $e_0$ belongs to $CR(T)$.
\end{fact}
\begin{proof}
	As in Fact~\ref{Fact:e0.CR}, we fix any $\delta>0$ and by Lemma~\ref{Lem:useful} we just have to find a $\delta$-chain for $T$ from the vector $e_0$ to the zero-vector $0_X \in X$ and another $\delta$-chain for $T$ from $0_X$ to $e_0$.
	\begin{enumerate}[--]
		\item \textbf{Step 1: from $0_X$ to  $e_0$}. If we find $m_1 \in \NN$ with $m_1>1$ and $1 < \delta \cdot |\mu_1|^{m_1-1}$, then the finite sequence $(f_l)_{l=0}^{m_1}$ in $X$ defined as
		\[
		\begin{cases}
			f_0 := 0_X, &\\[7.5pt]
			f_1 := \left( \tfrac{1}{\mu_1} \right)^{m_1-1} \cdot \ e_{m_1-1}, &\\[12.5pt]
			f_l := T^{l-1}(f_1) & \text{ for } l=2,...,m_1,
		\end{cases}
		\]
		can be easily checked to be a $\delta$-chain for $T$ from $0_X$ to $e_0$, exactly as in \textbf{Step 1} of Fact~\ref{Fact:e0.CR}.
		
		\item \textbf{Step 2: from $e_0$ to $0_X$}. For this case find $n \in \NN$ with $n>1$ and such that $|\mu_1|^{n} < \delta \cdot |\mu_2|^{n-1}$, and let $m_2:=2n-1$. Considering
		\[
		\begin{cases}
			f_0 := e_0, &\\[7.5pt]
			f_1 := T(f_0) - \left( \tfrac{1}{\mu_2} \right)^{n-1} \cdot \ \mu_1^{n} \ \cdot \ e_{(-n,n-1)}, &\\[10pt]
			f_l := T^{l-1}(f_1) & \text{ for } l=2,...,n,...,m_2-1, \\[7.5pt]
			f_{m_2} := 0_X,
		\end{cases}
		\]
		we have that $f_0=e_0$ and $f_{m_2}=0_X$ so we just have to check that $(f_l)_{l=0}^{m_2}$ is a $\delta$-chain. Indeed, since
		\begin{align*}
			f_n = T^{n-1}(f_1) &= T^n(f_0) - T^{n-1}\left( \left( \tfrac{1}{\mu_2} \right)^{n-1} \cdot \ \mu_1^{n} \ \cdot \ e_{(-n,n-1)} \right)\\
			&= \mu_1^{n} \cdot e_{-n} - \mu_2^{n-1} \cdot \left( \tfrac{1}{\mu_2} \right)^{n-1} \cdot \ \mu_1^{n} \cdot \left(e_{(-n,0)} + e_{-n}\right) = - \mu_1^{n} \cdot e_{(-n,0)}
		\end{align*}
		and hence
		\begin{align*}
			T(f_{m_2-1}) = T^{m_2-1}(f_1) &= T^{(m_2-1)-(n-1)}(T^{n-1}(f_1)) = T^{m_2-n}(f_n) = T^{n-1}(f_n)\\
			&= T^{n-1}\left( -\mu_1^n \cdot e_{(-n,0)} \right) = - \left( \tfrac{1}{\mu_2} \right)^{n-1} \cdot \ \mu_1^n \ \cdot \ e_{(-n,-n+1)},
		\end{align*}
		we deduce that $\| f_1 - T(f_0) \| = \| f_{m_2} - T(f_{m_2-1}) \| = \left( \tfrac{1}{|\mu_2|} \right)^{n-1} \cdot \ |\mu_1|^{n} \ < \ \delta$.
	\end{enumerate}
	The arbitrariness of $\delta>0$ together with \textbf{Step~1} and \textbf{Step~2} show that $e_0 \in CR(T)$.
\end{proof}

From this point one can show, using exactly the same arguments included after Fact~\ref{Fact:e0.CR}, that the vectors $e_{-n}$ and $e_n$ belong to $CR(T)$ for every $n \in \NN$. Alternatively, one can use that the equality between the sets of chain recurrent vectors $CR(T)=CR(T^{-1})$ holds since $T$ is invertible, and hence that $CR(T)$ is a closed $T$-and-$T^{-1}$-invariant linear subspace of $X$ (see \cite[Proposition~26]{BernardesPe2024_AM}). In fact,
\[
T^n\left( \tfrac{1}{\mu_1^n} \cdot e_0 \right) = e_{-n} \quad \text{ and } \quad T^{-n}\left(\mu_1^n \cdot e_0\right) = e_n,
\]
so one can directly conclude from the previous comments that $e_{-n}, e_n \in CR(T)$, and hence \eqref{eq:inclusion1T} holds.

\subsection[Calculating CR(T): the second inclusion]{Calculating $CR(T)$: the second inclusion}

We will complete the proof of Theorem~\ref{The:main} by showing that
\begin{equation}\label{eq:inclusion2T}
	CR(T) \subset \cl{\lspan\{ e_n \ ; \ n \in \ZZ \}}.
\end{equation}
Following again the strategy used in Section~\ref{Sec_3:non-invertible} we prove the following fact, which strongly depends on statement (b) of Lemma~\ref{Lem:useful} (this same idea is also crucially used in Theorem~\ref{The:classical} below):

\begin{fact}\label{Fact:non.CR.T}
	Let $f = (f(v))_{v\in V} \in X$. If $f \notin \cl{\lspan\{ e_n \ ; \ n \in \ZZ \}}$, then $f \notin CR(T)$.
\end{fact}
\begin{proof}
	By assumption there exists some $(k_0,j_0) \in \NN\times\ZZ$ such that $|f(-k_0,j_0)|>0$. Moreover, since the infinite matrix of weights $(\lambda_{(-k,j)})_{(k,j)\in\NN\times\ZZ} \in \KK^{\NN\times\ZZ}$ fulfills \eqref{eq:non.CR}, we have two possibilities:
	\begin{enumerate}[--]
		\item \textbf{Case 1: $\sum_{j=1}^{\infty} \left| \lambda_{(-k_0,1)} \cdots \lambda_{(-k_0,j)} \right| < \infty$}. In this case consider a positive value $\delta>0$ small enough to fulfill the inequality
		\begin{equation}\label{eq:delta1T}
			\delta \cdot \left( 1 + \sum_{j=1}^{\infty} \left| \lambda_{(-k_0,j_0+1)} \cdots \lambda_{(-k_0,j_0+j)} \right| \right) < \left| f(-k_0,j_0) \right|.
		\end{equation}
		If we assume by contradiction that $f \in CR(T)$, then by Lemma~\ref{Lem:useful} there exists a $\delta$-chain $(f_l)_{l=0}^m$ for $T$ from the zero-vector $0_X \in X$ to $f$ and hence a finite sequence $(g_l)_{l=1}^m$ in $X$ such that
		\[
		f = f_m = T^m(0_X) + \sum_{l=1}^m T^{m-l}(g_l) \quad \text{ and } \quad | g_l(v) | < \delta \text{ for all } 1\leq l\leq m \text{ and } v \in V.
		\]
		This implies that
		\begin{align*}
			f(-k_0,j_0) &= \sum_{l=1}^m [T^{m-l}(g_l)](-k_0,j_0) = \sum_{l=0}^{m-1} [T^{l}(g_{m-l})](-k_0,j_0)\\
			&= g_m(-k_0,j_0) + \sum_{l=1}^{m-1} \left( \lambda_{(-k_0,j_0+1)}\cdots\lambda_{(-k_0,j_0+l)} \right) \cdot g_{m-l}(-k_0,j_0+l)
		\end{align*}
		and hence by \eqref{eq:delta1T} we have that
		\[
		|f(-k_0,j_0)| \leq \delta \cdot \left( 1 + \sum_{l=1}^{m-1} \left| \lambda_{(-k_0,j_0+1)}\cdots\lambda_{(-k_0,j_0+l)} \right| \right) < |f(-k_0,j_0)|,
		\]
		contradicting the existence of the $\delta$-chain $(f_l)_{l=0}^m$.
		
		\item \textbf{Case 2: $\sum_{j=1}^{\infty} \left| \lambda_{(-k_0,-(j-1))} \cdots \lambda_{(-k_0,0)} \right|^{-1} < \infty$}. In this case consider $\delta>0$ satisfying the inequality
		\begin{equation}\label{eq:delta2T}
			\delta \cdot \sum_{j=1}^{\infty} \left| \lambda_{(-k_0,j_0-(j-1))}\cdots\lambda_{(-k_0,j_0)} \right|^{-1} < \left| f(-k_0,j_0) \right|.
		\end{equation}
		Assume again by contradiction that $f \in CR(T)$. Then by Lemma~\ref{Lem:useful} there exists a $\delta$-chain $(f_l)_{l=0}^m$ for $T$ from the vector $f$ to the zero-vector $0_X$ and hence a finite sequence $(g_l)_{l=1}^m$ in $X$ such that
		\[
		0_X = f_m = T^m(f) + \sum_{l=1}^m T^{m-l}(g_l) \quad \text{ and } \quad | g_l(v) | < \delta \text{ for all } 1\leq l\leq m \text{ and } v \in V.
		\]
		This implies that
		\begin{align*}
			0 &= f_m(-k_0,j_0-m) = \left( \lambda_{(-k_0,j_0-(m-1))}\cdots\lambda_{(-k_0,j_0)} \right) \cdot f(-k_0,j_0) + \sum_{l=1}^m [T^{m-l}(g_l)](-k_0,j_0-m)\\
			&= \left( \lambda_{(-k_0,j_0-(m-1))}\cdots\lambda_{(-k_0,j_0)} \right) \cdot f(-k_0,j_0) + \sum_{l=1}^m \left( \lambda_{(-k_0,j_0-(m-1))}\cdots\lambda_{(-k_0,j_0-l)} \right) \cdot g_l(-k_0,j_0-l)
		\end{align*}
		and hence by \eqref{eq:delta2T} we have that
		\begin{align*}
			|f(-k_0,j_0)| &\leq \left| \lambda_{(-k_0,j_0-(m-1))}\cdots\lambda_{(-k_0,j_0)} \right|^{-1} \cdot \sum_{l=1}^m \left| \lambda_{(-k_0,j_0-(m-1))}\cdots\lambda_{(-k_0,j_0-l)} \right| \cdot \left| g_l(-k_0,j_0-l) \right| \\
			&\leq \delta \cdot \sum_{l=1}^m \left| \lambda_{(-k_0,j_0-(l-1))}\cdots\lambda_{(-k_0,j_0)} \right|^{-1} < |f(-k_0,j_0)|,
		\end{align*}
		contradicting again the existence of the $\delta$-chain $(f_l)_{l=0}^m$ and finally proving that $f \notin CR(T)$.\qedhere
	\end{enumerate}
\end{proof}

\section{Classical backward shifts are not enough}\label{Sec_5:classical}

In this section we show that Question~\ref{Q:main} cannot be solved via classical (unilateral neither bilateral) weighted backward shifts (acting on $\NN$ and $\ZZ$ respectively). This fact follows from Theorem~\ref{The:classical} below, where we prove that a classical weighted backward shift acting on a Fr\'echet sequence space is a chain recurrent operator whenever it admits a non-zero chain recurrent vector.

We must strongly emphasize that Theorem~\ref{The:classical} is just a modest extension of the already existing and well-known characterizations for the chain recurrent behaviour of classical weighted backward shifts given in \cite[Theorems~20~and~22]{AlBerMess2021} and \cite[Theorems~13,~14,~15~and~16]{BernardesPe2024_AM}. Indeed, our contribution here is to present a very slight refinement of the arguments exhibited in \cite[Theorem~13]{BernardesPe2024_AM}, based on statement (b) of Lemma~\ref{Lem:useful}. As in \cite{BernardesPe2024_AM} we adopt the convention that $c/0 = \infty$ whenever $c \in \ ]0,\infty[$.

\begin{theorem}\label{The:classical}
	Let $X$ be a Fr\'echet sequence space over $V = \NN$ or $\ZZ$ in which the sequence of canonical vectors $(e_n)_{n \in V}$ is a Schauder basis. Assume that $(\|\cdot\|_k)_{k\in\NN}$ is an increasing sequence of seminorms inducing the topology of $X$, and that $\sbf{\lambda}=(\lambda_n)_{n\in V} \in \KK^{V}$ is a sequence of non-zero weights such that the respective (unilateral or bilateral) weighted backward shift
	\[
	B_{\sbf{\lambda}} : (f(n))_{n\in V} \in X \longmapsto (\lambda_{n+1} \cdot f(n+1))_{n\in V} \in X,
	\] 
	is a well-defined operator. Then the following statements are equivalent:
	\begin{enumerate}[{\em(i)}]
		\item $B_{\sbf{\lambda}}$ is a chain recurrent operator (that is, $CR(B_{\sbf{\lambda}})=X$);
		
		\item $B_{\sbf{\lambda}}$ admits a non-zero chain recurrent vector (that is, $CR(B_{\sbf{\lambda}}) \neq \{0_X\}$);
		
		\item for some (and hence for every) integer $n_0 \in V \cup \{0\}$ it holds that
		\begin{equation}\label{eq:+}
			\sum_{n=1}^{\infty} \frac{\left| \lambda_{n_0+1} \cdots \lambda_{n_0+n} \right|}{\left\| e_{n_0+n} \right\|_k} = \infty \quad \text{ when } V=\NN \text{ or } \ZZ,
		\end{equation}
		and
		\begin{equation}\label{eq:-}
			\sum_{n=1}^{\infty} \frac{1}{\left| \lambda_{n_0-(n-1)} \cdots \lambda_{n_0} \right| \cdot \left\| e_{n_0-n} \right\|_k} = \infty \quad \text{ when } V=\ZZ,
		\end{equation}
		for all $k \in \NN$.
	\end{enumerate}
\end{theorem}
The equivalence (i) $\Leftrightarrow$ (iii) was shown in \cite[Section~4]{BernardesPe2024_AM} for the particular case $n_0=0$, but let us briefly justify the ``{\em and hence for every}'' part of statement (iii): given a pair of integers $n_1,n_2 \in V\cup\{0\}$ with $n_1<n_2$ and such that \eqref{eq:+} holds for all $k \in \NN$ and $n_0=n_1$, then we can find some $k_0 \in \NN$ such that $\|e_{n_1+l}\|_k\neq 0$ for all $k\geq k_0$ and $1\leq l\leq n_2-n_1$, which implies \eqref{eq:+} for all $k\geq k_0$ and $n_0=n_2$ because the divergent behaviour of a series remains invariant when a finite quantity of non-infinite terms are removed, but also when all its terms are multiplied by the same non-zero number (in this case we would divide every term by $|\lambda_{n_1+1}\cdots\lambda_{n_2}|$). Thus, since the sequence of seminorms $(\|\cdot\|_k)_{k\in\NN}$ is increasing, it follows that \eqref{eq:+} holds for all $k \in \NN$ and $n_0=n_2$. A completely symmetrical argument shows that, if \eqref{eq:-} holds for all $k \in \NN$ and $n_0=n_2$, then \eqref{eq:-} holds for all $k \in \NN$ and $n_0=n_1$.
\begin{proof}[Proof of Theorem~\ref{The:classical}]
	The equivalence (i) $\Leftrightarrow$ (iii) follows from \cite[Theorems~14~and~16]{BernardesPe2024_AM} together with the previous comments and, since (i) $\Rightarrow$ (ii) is trivial, we just have to show (ii) $\Rightarrow$ (iii). Thus, assume that statement (ii) holds and let $f^* = (f^*(n))_{n\in V} \in X$ be a non-zero chain recurrent vector for the (unilateral or bilateral) weighted shift $B_{\sbf{\lambda}}$, so that there exists some integer $n_0 \in V\cup\{0\}$ fulfilling that $f^*(n_0+1)\neq 0$. Since $CR(B_{\sbf{\lambda}})$ is a linear subspace of $X$, we will assume that $f^*(n_0+1)=\frac{1}{\lambda_{n_0+1}}\cdot$
	
	From now on we slightly modify the proof of \cite[Theorem~13]{BernardesPe2024_AM} to obtain (iii) for $n_0$:
	
	We start by showing that \eqref{eq:+} holds in any of the cases $V = \NN$ or $\ZZ$ at the same time. Fix any positive integer $k \in \NN$ and note that we may assume $\|e_{n_0+n}\|_k \neq 0$ for all $n \in \NN$, since otherwise the desired equality would hold trivially. By the Banach-Steinhaus theorem, there exists $\delta>0$ such that
	\begin{equation}\label{eq:BS.theorem}
		g=(g(n))_{n\in V} \in X \text{ with } \|g\| < \delta \quad \Longrightarrow \quad \left\|g(n) \cdot e_n\right\|_k < 1 \text{ for all } n \in V,
	\end{equation}
	where $\|\cdot\|$ is the F-norm described in \eqref{eq:F-norm}. Given any $r>0$, by Lemma~\ref{Lem:useful} we can find a $\delta$-chain $(f_l)_{l=0}^m$ for $B_{\sbf{\lambda}}$ from $0_X$ to $r f^* = (r f^*(n))_{n\in V}$ and hence a finite sequence $(g_l)_{l=1}^m$ in $X$ such that
	\begin{equation}\label{eq:rf*}
		f_m = B_{\sbf{\lambda}}^m(f_0) + \sum_{l=1}^m B_{\sbf{\lambda}}^{m-l}(g_l) \quad \text{ and } \quad \left|g_{m-(l-1)}(n_0+l)\right| < \frac{1}{\left\|e_{n_0+l}\right\|_k} \text{ for every } 1\leq l\leq m,
	\end{equation}
	where these last inequalities come from \eqref{eq:BS.theorem}. Since $f_0=0_X$, $f_m=rf^*$ and $f^*(n_0+1)=\frac{1}{\lambda_{n_0+1}}$,
	\begin{align*}
		r = \lambda_{n_0+1} \cdot f_m(n_0+1) &= \lambda_{n_0+1} \cdot \left( 0 +  \sum_{l=1}^m [B_{\sbf{\lambda}}^{m-l}(g_l)](n_0+1) \right) = \lambda_{n_0+1} \cdot \sum_{l=1}^m [B_{\sbf{\lambda}}^{l-1}(g_{m-(l-1)})](n_0+1)\\
		&= \sum_{l=1}^m \left( \lambda_{n_0+1} \cdots \lambda_{n_0+l} \right) \cdot g_{m-(l-1)}(n_0+l),
	\end{align*}
	and the triangle inequality together with the inequalities from \eqref{eq:rf*} yield
	\[
	r \leq \sum_{l=1}^{m} \left|\lambda_{n_0+1} \cdots \lambda_{n_0+l}\right| \cdot \left|g_{m-(l-1)}(n_0+l)\right| < \sum_{l=1}^{m} \frac{\left| \lambda_{n_0+1} \cdots \lambda_{n_0+l} \right|}{\left\| e_{n_0+l} \right\|_k}.
	\]
	Since $r>0$ is arbitrary, \eqref{eq:+} holds in any of the cases $V=\NN$ or $\ZZ$.
	
	It remains to show \eqref{eq:-} when $V=\ZZ$. In this case we can fix $k \in \NN$ and assume that $\|e_{n_0-n}\|_k\neq 0$ for all $n \in \NN$, and we can choose again some $\delta>0$ such that \eqref{eq:BS.theorem} is fulfilled. Since the set $CR(B_{\sbf{\lambda}})$ is $B_{\sbf{\lambda}}$-invariant, given any $r>0$ and again by Lemma~\ref{Lem:useful} we can find a $\delta$-chain $(f_l)_{l=0}^m$ for $B_{\sbf{\lambda}}$ from the vector $r B_{\sbf{\lambda}}(f^*)$ to $0_X$ and hence a finite sequence $(g_l)_{l=1}^m$ in $X$ such that
	\begin{equation}\label{eq:rBf*}
		f_m = B_{\sbf{\lambda}}^m(f_0) + \sum_{l=1}^m B_{\sbf{\lambda}}^{m-l}(g_l) \quad \text{ and } \quad \left|g_l(n_0-l)\right| < \frac{1}{\left\|e_{n_0-l}\right\|_k} \text{ for every } 1\leq l\leq m,
	\end{equation}
	where these last inequalities come from \eqref{eq:BS.theorem}. Since $f_m=0_X$, $f_0=rB_{\sbf{\lambda}}(f^*)$ and $f^*(n_0+1)=\frac{1}{\lambda_{n_0+1}}$,
	\begin{align*}
		0 = f_m(n_0-m) &= r \cdot [B_{\sbf{\lambda}}^{m+1}(f^*)](n_0-m) + \sum_{l=1}^m [B_{\sbf{\lambda}}^{m-l}(g_l)](n_0-m)\\
		&= r \cdot \left( \lambda_{n_0-(m-1)} \cdots \lambda_{n_0} \right) + \sum_{l=1}^m \left( \lambda_{n_0-(m-1)} \cdots \lambda_{n_0-l} \right) \cdot g_l(n_0-l),
	\end{align*}
	and the triangle inequality together with the inequalities from \eqref{eq:rBf*} yield
	\[
	r \leq \frac{1}{\left| \lambda_{n_0-(m-1)} \cdots \lambda_{n_0} \right|} \cdot \sum_{l=1}^{m} \left| \lambda_{n_0-(m-1)} \cdots \lambda_{n_0-l} \right| \cdot \left| g_l(n_0-l) \right| < \sum_{l=1}^{m} \frac{1}{\left| \lambda_{n_0-(l-1)} \cdots \lambda_{n_0} \right| \cdot \left\| e_{n_0-l} \right\|_k}.
	\]
	Since $r>0$ is again arbitrary, \eqref{eq:-} follows for the case $V=\ZZ$, and the proof is complete.
\end{proof}

As we were advancing at the beginning of this section, Theorem~\ref{The:classical} shows that Question~\ref{Q:main} cannot be solved via classical weighted backward shifts. Indeed, if we consider a unilateral or bilateral weighted shift $B_{\sbf{\lambda}}:X\longrightarrow X$ as described in the statement of Theorem~\ref{The:classical}, then the restriction of $B_{\sbf{\lambda}}$ to $CR(B_{\sbf{\lambda}})$ can be either the trivial operator $B_{\sbf{\lambda}}\res_{\{0_X\}}:\{0_X\}\longrightarrow\{0_X\}$ or else the operator $B_{\sbf{\lambda}}$ itself, and in both cases $B_{\sbf{\lambda}}:CR(B_{\sbf{\lambda}})\longrightarrow CR(B_{\sbf{\lambda}})$ is again a chain recurrent operator.

Note also that, even if we consider classical shifts with some of the weights equal to $0$, this is not enough to solve Question~\ref{Q:main} above. Indeed, if under the assumptions of Theorem~\ref{The:classical} we admit the equality $\lambda_{n_0}=0$ for some index $n_0 \in V$, then one can show (similarly to Fact~\ref{Fact:non.CR}) that both
\[
Y_{n_0}^- := \cl{\lspan\{ e_n \ ; \ n<n_0 \}} \quad \text{ and } \quad Y_{n_0}^+ := \cl{\lspan\{ e_n \ ; \ n\geq n_0 \}}
\]
are closed $B_{\sbf{\lambda}}$-invariant subspaces of $X$ for which $CR(B_{\sbf{\lambda}}) \subset Y_{n_0}^+$. Using this last fact we have that:
\begin{enumerate}[--]
	\item if the set of indices $J := \{ n \in V \ ; \ \lambda_n=0 \}$ is not bounded from above, then $CR(B_{\sbf{\lambda}}) = \{0_X\}$;
	
	\item while if $J$ is bounded from above and $n_0 := \max(J)$, then Theorem~\ref{The:classical} shows that $CR(B_{\sbf{\lambda}})\neq \{0_X\}$, and in particular $CR(B_{\sbf{\lambda}})=Y_{n_0}^+$, if and only if
	\[
	\sum_{n=1}^{\infty} \frac{\left| \lambda_{n_0+1} \cdots \lambda_{n_0+n} \right|}{\left\| e_{n_0+n} \right\|_k} = \infty \quad \text{ for all } k \in \NN.
	\]
\end{enumerate} 
The restricted operator $B_{\sbf{\lambda}}\res_{CR(B_{\sbf{\lambda}})}:CR(B_{\sbf{\lambda}})\longrightarrow CR(B_{\sbf{\lambda}})$ would be chain recurrent in any case, and we have implicitly proved the following:

\begin{corollary}\label{Cor:classical}
	Let $X$ be a Fr\'echet sequence space over $V = \NN$ or $\ZZ$ in which the sequence of canonical vectors $(e_n)_{n \in V}$ is a Schauder basis. Assume that $(\|\cdot\|_k)_{k\in\NN}$ is an increasing sequence of seminorms inducing the topology of $X$, and that $\sbf{\lambda}=(\lambda_n)_{n\in V} \in \KK^{V}$ is a sequence of (possibly zero) weights such that the respective (unilateral or bilateral) weighted backward shift
	\[
	B_{\sbf{\lambda}} : (f(n))_{n\in V} \in X \longmapsto (\lambda_{n+1} \cdot f(n+1))_{n\in V} \in X,
	\] 
	is a well-defined operator. Then the set of chain recurrent vectors $CR(B_{\sbf{\lambda}}\res_{CR(B_{\sbf{\lambda}})})$ for the restricted operator $B_{\sbf{\lambda}}\res_{CR(B_{\sbf{\lambda}})}$ coincides with the set $CR(B_{\sbf{\lambda}})$, and hence $B_{\sbf{\lambda}}\res_{CR(B_{\sbf{\lambda}})}$ is chain recurrent.
\end{corollary}

We would like to emphasize the possible interest of Theorem~\ref{The:classical} from a theoretical point of view by comparing the notion of chain recurrence with that of hypercyclicity. Indeed, in the works \cite{Seceleanu2010} and \cite{ChanSe2012}, Chan and Seceleanu showed the following: {\em if a classical weighted backward shift on $\ell^p$, $1\leq p<\infty$, admits an orbit that has a non-zero limit point, then the weighted shift is hypercyclic}. This result was called the ``{\em zero-one law of orbital limit points}'', which applies to both unilateral and bilateral weighted shifts. Later, in the 2020 work \cite{BoGr2020}, this result was extended by Bonilla and Grosse-Erdmann via a new and simpler proof, which also showed that this remains true for every classical unilateral and bilateral shift acting on any Fr\'echet sequence space with an unconditional basis. In our chain recurrence context, Theorem~\ref{The:classical} may be considered as a kind of {\em zero-one law} for classical shifts and, as suggested by the reviewer of this paper, we would like to pose the following natural question:

\begin{problem}\label{P:zero-one.law}
	Is there a general class of operators (possibly more general than classical shifts) such that any operator $T \in \Lc(X)$ in this class is chain recurrent as soon as $CR(T)\neq\{0_X\}$?
\end{problem}

We finish this paper by addressing Problem~\ref{P:zero-one.law} for the case of finite-dimensional operators, that is, for real and complex matrices. In particular, we are about to show that a Jordan block, coming from the Jordan normal form of a matrix, is chain recurrent if and only if it accepts a non-zero chain recurrent vector (see Lemma~\ref{Lem:finite-dimension}). We know that the results below are known to some experts but, since we could not find a reference, we will present them here for the sake of completeness.

\begin{lemma}\label{Lem:finite-dimension}
	Let $d \in \NN$ and assume that $T \in \CC^{d\times d}$ is either a diagonal matrix or a Jordan block,
	\[
	T =
	\begin{pmatrix}
		\lambda & 0 & 0 & \cdots & 0 \\
		0 & \lambda & 0 & \cdots & 0 \\
		\vdots & \vdots & \ddots & \ddots & \vdots \\
		0 & 0 & 0 & \lambda & 0 \\
		0 & 0 & 0 & 0 & \lambda \\
	\end{pmatrix} 
	\quad \text{ or } \quad T =
	\begin{pmatrix}
		\lambda & 1 & 0 & \cdots & 0 \\
		0 & \lambda & 1 & \cdots & 0 \\
		\vdots & \vdots & \ddots & \ddots & \vdots \\
		0 & 0 & 0 & \lambda & 1 \\
		0 & 0 & 0 & 0 & \lambda \\
	\end{pmatrix}
	,
	\]
	for some $\lambda \in \CC$. Then, for the finite-dimensional operator $T:\CC^d\longrightarrow\CC^d$ induced by the matrix $T$, the following statements are equivalent:
	\begin{enumerate}[{\em(i)}]
		\item $T$ is a chain recurrent operator (that is, $CR(T)=\CC^d$);
		
		\item $T$ admits a non-zero chain recurrent vector (that is, $CR(T) \neq \{0_{\CC^d}\}$);
		
		\item the complex value $\lambda \in \CC^n$ is unimodular (that is, $|\lambda|=1$).
	\end{enumerate}
\end{lemma}
\begin{proof}
	The implication (i) $\Rightarrow$ (ii) is trivial. Moreover, to check (ii) $\Rightarrow$ (iii) note that if $|\lambda|\neq 1$ we then have that $\sigma_p(T)=\sigma(T)=\{\lambda\}$ does not intersect the unit circle $\{z \in \CC \ ; \ |z|=1 \}$ of $\CC$, so that $T$ is a hyperbolic operator (see \cite[Definition~4]{BerCiDaMessPu2018}) and hence $CR(T)=\{0_{\CC^d}\}$ by \cite[Corollary~2.11]{AntMantVar2022}. Finally, assume that statement (iii) holds, so that $|\lambda|=1$, and let us check (i). We will use Lemma~\ref{Lem:useful}: for a fixed vector $f = (f(1),f(2),...,f(d)) \in \CC^d$ and a fixed $\delta>0$, we will show the existence of a $\delta$-chain from $f$ to $0_{\CC^d}$, and then the existence of a $\delta$-chain from $0_{\CC^d}$ to $f$.
	
	First, since the last coordinate of $T(f)$ is $\lambda f(d)$ and $|\lambda|=1$, it is clear that making $\delta$-perturbations only in the last coordinate one can obtain a $\delta$-chain $(f_l)_{l=0}^{m_1}$ for $T$ from $f_0=f$ to a vector $f_{m_1}$ with last coordinate equal to zero. Since the last two coordinates of $T(f_{m_1})$ are $\lambda f_{m_1}(d-1)$ and $0$, respectively, making $\delta$-perturbations only in the coordinate $d-1$ (which keeps the last coordinate equal to zero), one can also obtain a $\delta$-chain $(f_l)_{l=m_1}^{m_2}$ for $T$ from $f_{m_1}$ to a vector $f_{m_2}$ with the last two coordinates equal to zero. Continuing this process one ends up with a $\delta$-chain $(f_l)_{l=0}^{m_d}$ for $T$ from $f_0=f$ to $f_{m_d}=0_{\CC^d}$.
	
	We also claim the existence of a $\delta$-chain for $T$ from $0_{\CC^d}$ to $f$. Indeed, the inverse operator $T^{-1}$ is an upper-triangular matrix with the value $\lambda^{-1}$ along its diagonal and $|\lambda^{-1}|=1$, so that the same arguments used above show the existence of a $\delta$-chain $(f_l)_{l=0}^{m}$ for $T^{-1}$ from $f$ to $0_{\CC^d}$. Looking at $(f_l)_{l=0}^{m}$ in reversed from $f_m=0_{\CC^d}$ to $f_0=f$, and since $T(T^{-1}(g))=g$ for all $g \in \CC^d$, the claim follows.
\end{proof}

\begin{theorem}\label{The:finite-dimension}
	Let $d \in \NN$ and consider any (real or complex) matrix $T \in \KK^{d\times d}$. Then, for the finite-dimensional operator $T:\KK^d\longrightarrow\KK^d$ induced by the matrix $T$, the following statements hold:
	\begin{enumerate}[{\em(a)}]
		\item $T$ is chain recurrent if and only if the roots over $\CC$ of its characteristic polynomial are unimodular.
		
		\item The restriction $T\res_{CR(T)}:CR(T)\longrightarrow CR(T)$ is chain recurrent, that is, $CR(T\res_{CR(T)})=CR(T)$.
	\end{enumerate}
\end{theorem}
\begin{proof}
	Assume first that $\KK=\CC$, so that $T$ is seen as complex-valued matrix. In this case we can assume that $T$ is given in its Jordan normal form, so that $\CC^d$ can be decomposed as a topological direct sum of finitely many $T$-invariant subspaces for which $T$ acts as the blocks described in Lemma~\ref{Lem:finite-dimension}. Both statements (a) and (b) follow now from Lemma~\ref{Lem:finite-dimension} and \cite[Proposition~34]{BernardesPe2024_AM}.
	
	When $\KK=\RR$, and $T$ is a real-valued matrix acting on $\RR^d$, we have two possibilities: one can use the real Jordan normal form, but then a real-version of Lemma~\ref{Lem:finite-dimension} has to be proved; or we can consider the complexification of $T$, which means studying $T$ as a real-valued matrix acting on $\CC^d$. For the second option we can apply the complex version of this result, and using \cite[Proposition~34]{BernardesPe2024_AM} we can conclude the proof by projecting the set of complex chain recurrent vectors into their real coordinate.
\end{proof}

The reader may note that Theorem~\ref{The:finite-dimension} shows, in particular, that Question~\ref{Q:main} can not be solved via finite-dimensional operators. Moreover, and as a last comment, we will compare Theorem~\ref{The:finite-dimension} above with \cite[Theorems~4.1 and 4.2]{CoMaPa2014} in which the ``just {\em recurrent}'' finite-dimensional operators were characterized. We focus on complex matrices for simplicity: given $T \in \CC^{d\times d}$ we have that
\begin{enumerate}[--]
	\item $T$ is {\em chain recurrent} if and only if $\sigma_p(T) = \sigma(T) \subset \{ z \in \CC \ ; \ |z|=1 \}$;
	
	\item and that $T$ is {\em recurrent} if and only if $\sigma_p(T) = \sigma(T) \subset \{ z \in \CC \ ; \ |z|=1 \}$ and $T$ is diagonalizable.
\end{enumerate}
Hence, even for finite-dimensional operators, the notion of {\em chain recurrence} is strictly weaker than that of {\em linear recurrence} as defined in the 2014 work \cite{CoMaPa2014}.

\section*{Funding}

The first author was supported by the Spanish Ministerio de Ciencia, Innovaci\'on y Universidades (grant number FPU2019/04094); also by MCIN/AEI/10.13039/501100011033/FEDER, UE Projects PID2019-105011GB-I00 and PID2022-139449NB-I00; and by the Fundaci\'o Ferran Sunyer i Balaguer.

\section*{Acknowledgments}

The authors would like to thank Emma D'Aniello and Martina Maiuriello for their warm hospitality during the \textit{45th Summer Symposium in Real Analysis} held in Caserta, June 2023, where this paper was initiated. Special appreciation is extended to Nilson C.\ Bernardes Jr.\ and Alfred Peris for engaging and insightful discussions on this topic, during and after this same conference. Finally, the authors would also like to thank the anonymous reviewer, whose careful reading and valuable observations have significantly improved this paper.

{\footnotesize $\ $\\

\textsc{Antoni L\'opez-Mart\'inez}: Universitat Polit\`ecnica de Val\`encia, Institut Universitari de Matem\`atica Pura i Aplicada, Edifici 8E, 4a planta, 46022 Val\`encia, Spain. e-mail: alopezmartinez@mat.upv.es\\

\textsc{Dimitris Papathanasiou}: Sabanci University, Orta Mahalle, \"Universitesi Cd. No: 27, 34956 Tuzla/Istanbul, Turkey. e-mail: d.papathanasiou@sabanciuniv.edu

}

\end{document}